\documentclass[12pt,twoside]{amsart}
\usepackage{amssymb}
\usepackage{a4wide}
\usepackage[latin1]{inputenc}
\usepackage[T1]{fontenc}
\usepackage{times}

\def\hpq0{h^{p,q}_{\leq 0}}
\def\Hpq0{\H_{\leq 0}^{p,q}}

\def\dbar{\bar\partial}
\def\ddbar{\partial\dbar}

\def\C{{\mathbb C}}

\def\D{\mathcal{D}}

\def\H{{\mathcal H}}

\def\Re{{\rm Re\,  }}

\def\be{\begin{equation}}
\def\ee{\end{equation}}

\newtheorem{thm}{Theorem}[section]
\newtheorem{lma}[thm]{Lemma}
\newtheorem{cor}[thm]{Corollary}
\newtheorem{prop}[thm]{Proposition}

\theoremstyle{definition}

\newtheorem{df}{Definition}

\theoremstyle{remark}

\newtheorem{preremark}{Remark}
\newtheorem{preex}{Example}

\numberwithin{equation}{section}

\title[]
{Lelong numbers and vector bundles.}

\address{Department of Mathematics\\Chalmers University
  of Technology \\
 S-412 96
  G\"OTEBORG\\SWEDEN} 

\email{ bob@chalmers.se}

\author[]{Bo Berndtsson}

\begin{document}

\begin{abstract}
We study Lelong numbers and integrability indices for $S^1$-invariant singular metrics on vector bundles over the disk.\end{abstract}

\maketitle
{\small \it Dedicated to the memory of G. M. Henkin.}
\section{Introduction.}

The notion of a singular metric on holomorphic line bundles has proved to be very useful in complex analysis and algebraic geometry. The analogous notion for holomorphic vector bundles, which appears in\cite{de Cataldo}, \cite{Berndtsson-Paun}, \cite{Raufi} and  \cite{Paun-Takayama}, see also  \cite{Hacon-Schnell} for a clear discussion,  has been much less studied. In this paper we will focus on a very particular situation: Vector bundles over the unit disk, $\Delta$,  with metrics that satisfy a condition of $S^1$-invariance. Although this is a very special situation we will show in two examples that such metrics occur naturally and can be useful in applications. In both these examples the principle idea is that one embeds a certain problem, or estimate, into a scale of problems depending on a real parameter $t$. Thinking of $t$ as the real part of a complex variable $\zeta$, the norms one wants to study can be seen as defining a hermitian metric on a (trivial) vector bundle of positive or negative curvature over a half plane, or equivalently, via a logarithmic map, the disk. Of particular interest is the behaviour of the norms as the parameter  tends to infinity, or the origin in the disk picture, where a singularity of the metric appears.  

As a bundle, $E$,  over the disk is necessarily trivial, we can write $E=V\times\Delta$, where $V$ is an $n$-dimensional complex vector space.   We  fix one such global trivialization, and it is should be stressed that our later discussion depends on the choice of trivialization. A singular metric on $E$ is then basically a measurable map $z\to \langle\cdot,\cdot\rangle_z$, where $\langle \cdot,\cdot\rangle_z$ is an hermitian form on the fiber $E_z=V$. We  want to allow that our quadratic forms are degenerate and also that they  may attain the value $+\infty$. To avoid a detailed discussion of what this means, we will here use the following definition of a singular metric.

\medskip

{\bf Definition:} A singular metric on the holomorphic vector bundle $E$ is a map from $\Delta$ to the space of positive definite quadratic forms on $V$, {\it which is defined almost everywhere}, and satisfies the condition that for any local holomorphic section $u$ of $E$, $\log\langle u,u\rangle=\log\|u\|^2$ is measurable and locally integrable.

\qed

(In practice, since we will discuss only metrics that have either positive or negative curvature, they will have much stronger regularity properties than just being measurable.)

Following the articles cited above we say that $h$ is negatively curved if for any holomorphic map $z\to u(z)\in V$, the norm function $\|u(z)\|^2_z$ is subharmonic as a function of $z$. We refer to \cite{Raufi}, for a detailed discussion of this. Here we just stress that the condition is equivalent to the seemingly stronger condition that $\log\|u(z)\|_z$ is subharmonic. (This follows since, if $\|u(z)\|^2_z$ is subharmonic for any choice of $u$, then $\|u e^{p(z)}\|_z^2=\|u\|^2_z e^{2\Re p(z)}$ is also subharmonic for any holomorphic function $p$, and this implies that $\log\|u\|_z$ is subharmonic.) We also say that our metric is positively  curved, if the induced metric on the dual bundle $E^*= V^*\times\Delta$ is negatively curved. It is a standard fact that for smooth and  strictly definite metrics, these notions mean that the curvature operators $\Theta^h=\dbar h^{-1}\partial h$ are positive. These definitions extend to bundles over a higher dimensional base.

One important feature of the definitions is that they define negativity and positivity of the curvature, without defining any sort of curvature form or current. Indeed, as shown by a pertinent example of \cite{Raufi}, the curvature of a singular metric, defined in a way analogous to how it is defined in the smooth case, will in general not be a current with measure coefficients, even if the metric has negative curvature. This is in stark contrast to the line bundle case ($n=1$), where $h=e^{-\phi}$ has the curvature current $\ddbar\phi$, which has measure coefficients if $h$ is positively or negatively curved, i e if $\phi$ is subharmonic or superharmonic. 

We will introduce and study a notion of Lelong number and integrability index for metrics on vector bundles. Recall that for a plurisubharmonic function $\phi$ in the ball of $\C^m$, the (classical) Lelong number of $\phi$ at the origin can be defined as 
$$
\gamma_\phi(0)=\liminf_{z\to 0}\phi(z)/\log|z|.
$$
On the other hand, the integrability index, $\iota_\phi(0)$, is defined as the infimum of all numbers $c>0$ such that 
$$
\int_0 e^{-2\phi/c} <\infty.
$$
As is well known, when $m=1$, the integrability index equals the Lelong number and this common value is also equal to $\mu(\{0\})$, where $\mu=(i/\pi)\ddbar\phi$, i e the 'curvature' of $ e^{-\phi}$ at the origin.

If $\phi$ is (pluri)subharmonic in $D$, $h=e^{-\phi}$ is a (singular) metric of positive curvature on the trivial line bundle $\C\times D$, whereas $e^\phi$ is a metric of negative curvature on the same (or actually dual)  bundle, dual to $e^{-\phi}$. The Lelong numbers are defined in terms of the asymptotic behaviour of $\log\|1\|^2_{e^\phi}$ as we approach zero, while the integrability index is defined in terms of $\|1\|^2_{e^{-\phi}}$. Similarily, our definition of Lelong numbers for vector bundles applies  to metrics of negative curvature, while the 'integrability indices' for vector bundles will be defined for metrics of positive curvature.

\begin{df} Let $h$ be a (singular) metric of negative curvature on the vector bundle $E=V\times \Delta$. Then the Lelong number at 0 of $h$ in the direction $u\in V$ is defined as
$$
\gamma_h(u,0):=\liminf_{z\to 0} \log\|u\|^2_z/\log |z|=\gamma_{\log\|u\|^2}(0).
$$
\end{df}

Later it will be more natural to consider the negative of the Lelong number, so we also put $\alpha(u)=-\gamma (u,0)$ (where we have suppressed the dependence on $h$). Thus,
$$
\alpha(u)=\limsup_{z\to 0}\log\|u\|^2_z/\log(1/|z|).
$$
Since for any norm $\|u_1+u_2\|\leq 2\max( \|u_1\|, \|u_2\|)$, we   see  that for any $\alpha$, 
$$
V_\alpha:=\{u\in V; \alpha(u)\leq \alpha\}
$$
is a linear  subspace of $V$.

Since $V$ has finite dimension $n$, the dimension  of the spaces $V_\alpha$ can only jump at (at most) $n$ places. In other words, there are numbers $\alpha_1\leq \alpha_2\leq ...\alpha_n$ such that 
$$
V_{\alpha_1}\subseteq V_{\alpha_2}\subseteq ... V_{\alpha_n}=V,
$$
and
if $\alpha_j\leq \alpha <\alpha_{j+1}$, then $V_\alpha=V_{\alpha_j}$, and $V_{\alpha_j}$ has dimension $j$. We also let $\alpha_0=-\infty$ and $V_{\alpha_0}=\{0\}$. 

This filtration of $V$ gives rise to a dual filtration of the dual space $V^*=:F$. Letting $ F_\alpha=V_\alpha^\perp$ be the space of vectors in $F$ that are annihilated by $V_\alpha$, we have 
$$
0=F_{\alpha_n}\subseteq ...F_{\alpha_2}\subseteq F_{\alpha_1}\subset F_{\alpha_0}=F.
$$

At this point we remark that if we did have a curvature form of the metric with measure coefficients, $\Theta$, then the Hermitian operator $\Theta(\{0\})$ would give us a decomposition of $V$ into eigenspaces. As we have seen, such a curvature operator does not always exist, but what we get instead  is a filtration of $V$.

The general idea is  that while the spaces $V_\alpha$ are defined in terms of Lelong numbers, the spaces $F_\alpha$ are characterized by (a version of ) the integrability index. We will now  try to make this precise under the additional assumption that our metric $h$ is $S^1$-invariant. By this we  mean that if $u\in V$, then $\|u\|_z =\|u\|_{z e^{i\theta}}$ for any $\theta$ in the circle $S^1$. 

Before we state our  result, we transfer our problem to the right half plane $\H=\{\Re \zeta >0\}$ via the exponential map $z= e^{-\zeta}$. Changing notation, we call $E=V\times\H$, and consider metrics on $E$ that only depend on  $t=\Re\zeta$. A metric on the bundle $E$ over $\H$ corresponds under the exponential map to a metric over the punctured disk, and in case the metric has negative curvature it extends to metric of negative curvature over the full disk if and only if the metric over the half plane is bounded as $t$ tends to infinity. It will be convenient to allow also a slightly more general situation. We say that a negatively curved metric over $\H$ has moderate growth (at infinity) if for some constants $a$ and $C$, and all $u$ in $V$, 
$$
\|u\|^2_t\leq C(u) e^{at}.
$$
Equivalently,  the metric has moderate growth if for some $a$, the corresponding metric over the punctured disk is such that $\|u\|^2_z |z|^a$ extends to a negatively curved metric over the full disk.

If $u$ is a vector in $V$, then $\log\|u\|^2_t$ is now a convex function, and the negative of our  Lelong numbers are
$$
\alpha(u)=\lim_{t\to\infty} (1/t)\log\|u\|^2_t.
$$
Notice that, by convexity, $(1/t)(\log\|u\|^2_t-\log\|u\|^2_0)$ is increasing, so the limit exists.

Our main result is as follows. 
\begin{thm} 
Let $\|\cdot\|_t$ be a metric of negative curvature and moderate growth on the bundle $E=V\times \H$ over $\H$, depending only on $t=\Re \zeta$ for $\zeta\in \H$. Let $\|\cdot\|_{-t}$ be the dual metric on $E^*=V^*\times\H$, and let $V_\alpha$, $F_\alpha$ be the filtrations of $V$ and $V^*=F$ described above. Let $\alpha_j\leq \alpha <\alpha_{j+1}$.  Then the following are equivalent for $v\in V^*$:

a.   $$v\in F_{\alpha_j},$$

b. $$\int_0^\infty \|v\|^2_{-t} e^{t\alpha} dt <\infty,$$

c. $$\limsup_{t\to\infty} (1/t)\log\|v\|^2_{-t}\leq -\alpha_{j+1}.$$
\end{thm}

The equivalence between conditions (a) and (b)  here says that the spaces $F_\alpha$ can be defined in terms of the integrability properties of the dual norms.  
In the next section we will give a proof of this somewhat technical looking result, the only non trivial part being the implication from a. to c. After that we will give two examples of how metrics satisfying our very particular assumptions actually arise in 'practice'.Notice however already now that it follows from Theorem 1.1 that for any $v$ in $V^*$, the set of $\alpha$ such that 
$$
\int_0^\infty \|v\|^2_{-t} e^{t\alpha} dt <\infty,
$$
is always an open   interval.

\section{ The proof of theorem 1.1}

We let $E=V\times \H$ be a trivial vector bundle over the right half plane. We suppose given a metric $\|u\|_t$ defined for $u$ in $V$ and $t>0$, such that the bundle metric $\|u\|_{\Re \zeta}$ has negative curvature. We also assume that the metric has moderate growth at infinity, so that $\|u\|^2_t\leq C(u) e^{at}$ for some constants $C$ and $a$.  It will be convenient to have the metric defined also for $t=0$. By the curvature assumption, $\log\|u\|_t$ is convex, hence also continuous and we assume it extends continuously to $t=0$. Then it follows from the convexity that 
$$
(1/t)\log(\|u\|_t/\|u\|_0)
$$
is an increasing function.  Hence $u$ lies in $V_\alpha=\{u; \alpha(u)\leq \alpha\}$ if and only if $\|u\|_t^2\leq\|u\|^2_0e^{t\alpha}$ for all $t\geq 0$.

At several occasions we will have use for the following lemma (see \cite{Coifman-Semmes}, \cite{Berman-Keller}, \cite{Lempert})
\begin{lma}
  Let $E$ be a vector bundle over a bounded domain $D$ in $\C$ (or $\C^n$). Let $h_1$ and $h_2$ be two metrics on $E$ that extend continuously to the closure of $D$. Assume that $h_1$ has semipositive curvature and $h_2$ has seminegative curvature and that $h_1\geq h_2$ over $\partial D$. Then $h_1\geq h_2$ in $D$.
\end{lma}

Take $t>0$. By the spectral theorem we can find a basis $e_j=e_j(t)$ which is orthonormal for $t=0$ and diagonalizes $\|\cdot\|_t$. Hence there are numbers $\lambda_j=\lambda_j(t)$, such that if $u=\sum c_j e_j$ then 
$$
\|u\|^2_0=\sum |c_j|^2,
$$
and
$$
\|u\|^2_t=\sum |c_j|^2 e^{t\lambda_j(t)}.
$$
Reordering, we assume that $\lambda_1\leq\lambda_2\leq ...\lambda_n$.

Given $t$ we now define a new metric
\be
\|u\|^2_{s,t}=\sum |c_j|^2 e^{s\lambda_j(t)}
\ee
for $0\leq s\leq t$. This metric is clearly flat and coincides with our original metric for $s=0$ and $s=t$. By the lemma it follows that $\|u\|^2_{s,t}\geq \|u\|^2_s$ for $0\leq s\leq t$. The min-max formula for eigenvalues then gives that $\lambda_j(s)\leq\lambda_j(t)$, so the $\lambda_j(t)$:s are increasing functions of $t$. Since our metric has moderate growth they are also bounded from above and therefore have limits $\lambda_j(\infty)$. We shall prove that $\lambda_j(\infty)=\alpha_j$.

If $\phi$ is a  subharmonic function in the disk, bounded from above, we can write it as $\phi=v+g$, where $v$ is harmonic and $g=G[\Delta\phi]$ is the Green potential of $\Delta\phi$. Replacing $\mu=\Delta\phi$ by a Dirac mass at 0, $\mu(\{0\})\delta_0$, we get a subharmonic function $\phi_0$ which is larger that $\phi$, harmonic in the punctured disk, with the same Lelong number at 0 as $\phi$. The next proposition generalizes this construction to the vector valued case. 
\begin{prop} Let $E=V\times \H$ have  a negatively curved  metric $h$ of moderate growth, that depends only on $t=\Re \zeta$. Then there is a  unique     negatively curved metric $h_\infty\geq h$, such that $h_\infty$ depends only on $t$, $h_\infty=h$ for $t=0$,  which is flat  and
  satisfies
  \be
  \lim_{t\to\infty}(1/t)\log\|u\|^2_{t, \infty}=\lim_{t\to\infty}(1/t)\log\|u\|^2_t
  \ee
  for any $u$ in $V$. 
Hence $h_\infty$   defines the same jumping numbers $\alpha_j$ and filtration $V_\alpha$ as $h$. 
\end{prop}
\begin{proof} Fix first $t>0$ and define a metric $h_t$ by
$$
\|u\|^2_{s,t}
$$
which is equal to $\|u\|^2_s$ if $s\geq t$ and flat for $0<s<t$. For $s<t$ it is given by (2.1) and we have seen that $\|u\|^2_s\leq \|u\|^2_{s,t}$ for all $s>0$. 

Since  $h_t$
is given by a max construction, $\log\|u(z)\|_{\Re z,t}$ is subharmonic for any holomorphic section $u(z)$ of $E$, so $h_t$ is negatively curved. 
It therefore follows from the lemma again that $h_t$ increases with $t$.  
Hence $h_t$  has a limit, $h_\infty$.  By the subharmonicity criterion, this limit has seminegative curvature and the
dual metric on $E^*$ also has seminegative curvature. Hence the metric is in fact flat.

If $u\in V_\alpha$ then $\|u\|^2_t\leq \|u\|^2_0 e^{t\alpha}$. Since $\log\|u\|^2_{s,t}$ is convex and $\log(\|u\|^2_0 e^{s\alpha})$ is linear it follows  that
$$
\|u\|^2_{s,t}\leq \|u\|^2_0e^{s\alpha}
$$
for $s<t$. Hence
$$
\|u\|^2_{s,\infty}\leq \|u\|^2_0 e^{s\alpha}
$$
so
$$
\lim_{s\to \infty} (1/s)\log\|u\|^2_{s,\infty}\leq \alpha.
$$
Hence the 'negative   Lelong numbers'  of $\|\cdot\|_{s,\infty}$ are not larger than the negative Lelong numbers of $\|\cdot\|_s$, which proves the existence part of the theorem since the opposite inequality is evident.

To prove the uniqueness part we first note that if we apply the first part of the argument to a metric that is already flat, then $\|\cdot\|_{s,t}=\|\cdot\|_s$ for $s\leq t$, which means that $\lambda_j(s)=\lambda_j(t)$ is independent of $s$. Moreover, since the span of $\{e_j(s)\}_{j\leq k}$ is equal to the span of $\{e_j(t)\}_{j\leq k}$ (or can be chosen equal in case some of the eigenvalues are multiple), it follows that the unitary change of basis from $e_j(s)$ to $e_j(t)$ is diagonal. This is because the corresponding unitary matrix is lower triangular, which implies that it is diagonal if it is unitary. It follows that we can choose the basis $e_j$ independent of $s$, so the norms can be written
$$
\sum |c_j|^2 e^{t\lambda_j}
$$
for fixed $\lambda_j$ and a fixed basis if $\|\cdot\|^2_t$ is  flat.

Now let $|u|_t$ be a flat metric satisfying the assumptions of the theorem. Then $ \|u\|_{t,\infty}\leq |u|_t$. By the discussion above, there are two fixed bases, orthnormal for $\|\cdot\|_0$,  $f_j$ and $\tilde f_j$ such that if $u=\sum a_j f_j=\sum b_j \tilde f_j$, then
$$
\|u\|^2_{t,\infty}=\sum |a_j|^2 e^{t\lambda_j(\infty)}, \quad |u|_t^2=\sum|b_j|^2 e^{t\tilde \lambda_j}.
$$
The jumping numbers for both these norms are  the same since both coincide with the jumping numbers for $\|u\|_t$. Hence $\lambda_j(\infty)=\tilde\lambda_j$. It also follows that the spans of $\{f_j\}_{j\leq k}$ and  $\{\tilde f_j\}_{j\leq k}$ are the same, so by the same argument as above the unitary base change matrix is diagonal. Hence the norms are identical, which is what we wanted to prove. 
\end{proof}

From the proof above we have that there is one fixed basis such that if  $u=\sum a_j f_j$, 
$$
\|u\|^2_{s,\infty}=\sum |a_j|^2 e^{s\lambda_j(\infty)}.
$$
Since for fixed $s$ the  norms $\|\cdot\|_{s,t}$ increase to $\|\cdot\|_{s,\infty}$, the corresponding  eigenvalues (with respect to $\|\cdot\|_0$) converge. Hence  $\lambda_j(\infty)=\lim_{t\to\infty} \lambda_j(t)$. On    the other hand, $\lambda_j(\infty)$ are the jumping numbers for $h_\infty$, which are the same as the jumping numpers of $h$, so $\lim_{t\to\infty}\lambda_j(t)=\alpha_j$. 

The next lemma is the crucial step in the proof of Theorem 1.1.
\begin{lma} Assume that $v\in V^*$ and that $v\in F_{\alpha_m}$, where $F_{\alpha_m}=V_{\alpha_m}^\perp$ and $\alpha_m<\alpha_{m+1}$. Then
  $$
  \limsup_{t\to \infty}(1/t)\log\|v\|^2_{-t}\leq-\alpha_{m+1}.
  $$
  \end{lma}
\begin{proof} We will carry out the proof assuming that all jumping numbers $\alpha_j$ are different, leaving the mostly notational changes for the general case to the reader.

  Take $t>>0$ and write for $u\in V$, $u=\sum c_j e_j$ with respect to an orthonormal basis for $\|\cdot\|_0$ as above, so that
  $$
  \|u\|^2_t=\sum |c_j|^2 e^{t\lambda_j(t)}.
    $$

    We identify $V$ with $F=V^*$, by the conjugate linear isomorphism defined by $\|\cdot\|_0$, so that if $v\in F$, $v=\sum v_je_j$ and the dual norm of $v$ is
    $$
    \|v\|^2_{-t}=\sum |v_j|^2 e^{-t\lambda_j(t)}.
    $$
    We can assume $\|v\|_0=1$. 
    We also use the orthonormal basis $f_j$, with respect to which $u=\sum b_jf_j$ and 
    $$
    \|u\|^2_{t,\infty}=\sum |b_j|^2e^{t\alpha_j}.
    $$
    Write $f_j=\sum d^j_k e_k$. Then
    $$
    e^{t\alpha_j}=\|f_j\|^2=\sum |d^j_k|^2 e^{t\lambda_k(t)}.
    $$
    If $t$ is large enough, $\lambda_k(t)$ is close to $\alpha_k$ . Taking $j=1$, we see that $d^1_k$ is then (exponentially) small if $k>1$, so $|d^1_1|$ must be close to 1 (since the change of basis is unitary). We can therefore subtract suitable multiples of $f_1$ from all the $f_j$ with $j>1$, to achieve $d^j_1=0$. Carrying on in this way and renormalizing we get a new basis $g_j$ of unit  vectors, such that $g_1=f_1$ and the span $[g_1, ...g_j]=V_{\alpha_j}$ and
    $$
    g_j=\sum g^j_ke_k,
    $$
    where the matrix $(g^j_k)$ is upper triangular so that $g^j_k=0$ if $k<j$. 
    We then have that $v\perp g_j$ for $j\leq m$ by hypothesis. We claim that for any $\epsilon>0$, and any $k$, 
    \be
    |v_k|^2e^{-t\lambda_k(t)}\leq Ce^{-t\lambda_{m+1}} e^{\epsilon t},
    \ee
    if $t$ is large enough. This is clearly true, with $C=1$ and $\epsilon=0$, if $k>m$, since $\lambda_k\geq\lambda_{m+1}$ then, and $|v_k|\leq 1$. We next argue by induction, assuming the claim holds for all larger indices than $k$. 
    Since $v\perp g_k$
    $$
    \bar v_k g^k_k=-\sum_{j>k} \bar v_jg^k_j.
    $$
    Since $g_k\in V_{\alpha_k}$ and $\|g_k\|_0=1$ we have $\|g_k\|^2_t\leq e^{t\alpha_k}$, so
    $$
    \sum |g^k_j|^2 e^{t\lambda_j(t)}\leq e^{t\alpha_k}.
    $$
    If $j>k$, $\lambda_j$ is close to $\alpha_j>\alpha_k$ so $g^k_j$ must be (exponentially) small, and then $|g^k_k|$ must be close to 1, since $g^k_j=0$ if $j<k$. Hence
    $$
    |v_k|^2\leq C\sum_{j>k}|v_j|^2 e^{-t\lambda_j(t)}\sum_{j>k}|g^k_j|^2e^{t\lambda_j(t)}\leq
    C\sum_{j>k}|v_j|^2e^{-t\lambda_j(t)} e^{t\alpha_k}.
    $$
    By the induction hypothesis this is smaller than
    $$
    C e^{-t\lambda_{m+1}} e^{\epsilon t} e^{t\alpha_k}\leq C e^{-t\lambda_{m+1}} e^{\epsilon t} e^{t\lambda_k(t)}e^{\epsilon t},
    $$
    is $t$ is large enough, which  proves the claim. 
 
    From the claim it follows immediately that
    $$
    \|v\|^2_{-t}\leq Ce^{-t\alpha_{m+1}} e^{\epsilon t},
    $$
    so the proposition is proved.

    \end{proof}

We are now ready to prove Theorem 1.1. It is immediately clear that (b) implies that 
$$
\liminf \|v\|^2_{-t} e^{t\alpha}=0.
$$
Take $u$ in $V_{\alpha_j}$ such that $\|u\|^2_0=1$ . Then 
$$
|\langle v,u\rangle|^2\leq \|v\|^2_{-t}\|u\|^2_t\leq\|v\|^2_{-t}e^{t\alpha_j},
$$
since $u\in V_{\alpha_j}$ and $\|u\|^2_0=1$ implies $\|u\|^2_t\leq e^{t\alpha_j}$. 
The liminf of the right hand side is zero. Hence $v$ lies in $F_{\alpha_j}$ so we have proved (a). This implies (c) by the lemma, which in turn gives (b) again. Hence all the conditions are equivalent and we have proved Theorem 1.1.

\section{ Example 1: The global case of the 'strong openness problem'.}
The original 'openness conjecture' concerns the local integrability of $e^{-\psi}$ where $\psi$ is plurisubharmonic and states that the interval of all positive numbers $p$ such that $e^{-p\psi}$ is integrable in some neighbourhood of the origin  is open. This was proved in \cite{Berndtsson-open}. The strong openness conjecture is the same statement for functions $|f|^2 e^{-p\psi}$, where $f$ is holomorphic, and was first proved by Guan-Zhou   in \cite{Guan-Zhou}. In this section we shall show how a (simpler) global version
of  strong openness  follows from  Theorem 1.1. The full strong openness would follow from an extension of Theorem 1.1 to bundles of infinite rank -- this is one reason why we think that such an extension would be interesting. (See the remark below.) 

We consider the following setting. $X$ is a compact projective (or only K\"ahler) manifold and $L$ is a semipositive holomorphic line bundle over $X$. Let $ e^{-\phi}$ be a smooth metric of semipositive curvature on $L$. Let $\omega=i\ddbar\phi$ and let $\psi$ be a function such that $2\psi$ is  $\omega$-plurisubharmonic function on $X$. This means that $\psi$ is integrable and $2i\ddbar\psi+\omega\geq 0$. (The constant 2 here will be clear later: it could be replaced by any number greater than 1.)

We consider the vector space $H^0(X, K_X+L)$ which we think of as the space of holomorphic $(n,0)$-forms on $X$ with values in $L$. It can be equipped with the $L^2$-norms
$$
c_n\int_X v\wedge\bar v e^{-\phi}
$$
and
$$
c_n\int_X v\wedge\bar v e^{-\phi-\psi}.
$$
( Here $c_n=i^{n^2}$ is the standard unimodular constant that makes the norms nonnegative.)
\begin{prop} Let $v$ be  an element of $H^0(X, K_X+L)$. Assume  that 
$$
c_n\int_X v\wedge\bar v e^{-\phi-\psi} <\infty.
$$
Then there is a number $p>1$ such that
$$
c_n\int_X v\wedge\bar v e^{-\phi-p\psi} <\infty.
$$
\end{prop}

The proof of the proposition is mimicked on the arguments in \cite{Berndtsson-open}. The first ingredient is the following calculus lemma which can be proved by direct computation.
\begin{lma} Let $x\leq 0$ and $0<p<2$. Then 
$$
\int_0^\infty e^{ps} e^{-2\max(x+s,0)} ds +1/p= C_{p} e^{-px}.
$$
\end{lma}
We next let $W:=H^0(X, K_X+L)$ and let $F=W\times\H$ be the trivial vector bundle over the right half plane with fiber $W$. We equip $W$ with the norms
$$
\|v\|^2_{-s}:=c_n\int_X v\wedge\bar v e^{-\phi-2\psi_s},
$$
where $\psi_s=\max(\psi+s,0)$. We normalize so that $\psi\leq 0$. Since $\phi+2\psi_s$ is plurisubharmonic on $X\times\H$ ( $s=\Re \zeta$) it follows from the results in \cite{Berndtsson} that $\|\cdot\|_{-\Re\zeta}$ defines a positively curved metric on our bundle $F$. By Lemma 3.2  we have that
$$
C_P c_n\int_X v\wedge\bar v e^{-\phi-p\psi}=\int_0^\infty \|v\|^2_{-s} e^{ps}ds -1/p.
$$

By Theorem 1.1,  the set of $p<2$ such that the right hand side here is finite,  is open, which proves Proposition 3.1.

\subsection{Local strong openness.}

One would perhaps  wish for a version of this proof for the local strong openess. We would then have a negative plurisubharmonic function $\psi$ in the unit ball, $B$, and $v$ a holomorphic function in the ball,  and study the integrability of $|v|^2 e^{-p\psi}$ over smaller balls. Approaching this problem in the same way we put $W=A^2(B)$, the Bergman space of square integrable holomorphic functions in the ball,
$$
\int_B |v|^2 <\infty.
$$
Given $\psi$, we define $\psi_s=\max(\psi+s,0)$ as before and introduce the norms
$$
\|v\|^2_{-s}:=\int_{B/2} |v|^2 e^{-2\psi_s}
$$
on $W$. One could then continue in the same way, given an extension of Theorem 1.1 to infinite dimensional spaces. This poses some obvious problems, but there is at least one feature of this set up that may be helpful.

We define $V$ as the dual of $W$ and  consider the norms $\|u\|_t$ on $V$ dual to $\|\cdot\|_{-t}$. We can then define the negative Lelong numbers $\alpha(u)$ as before, but probably the possible values of $\alpha$ will no longer form a finite set. They will however satisfy a {\it Noetherian property}: All decreasing sequences $V_{\alpha_j}$ are stationary for $j$ large. This means that if $\alpha=\alpha(u)$ for some $u$ in $V$, then there is some $\epsilon>0$ such that $V_{\alpha+\epsilon}=V_\alpha$.

  To see this, we note that all elements $u$ in $V$ with $\|u\|_s<\infty$ are bounded as functionals on $W$ by the square norm of a function $v$ on $B/2$. They therefore extend as linear functionals on the space of functions that are holomorphic only in a neighbourhood of the closure of $B/2$. If $h$ is holomorphic in a neighbourhood of $\bar B/2$  we can then define $hu$ by duality. It  follows that $\|hu\|_s\leq \sup_{B/2}|h| \|u\|^2_s$, so the subspaces $V_\alpha$ are stable under such multiplication. Hence the space of functions holomorphic in a neighbourhood of the closure of $B/2$ that are annihilated by $V_\alpha$ is a module over the ring $H(\bar B/2)$, and our claim follows from the Noetherian property of such modules. 

  The statement of Theorem 1.1 still makes sense in the infinite rank case, if we add the assumption that the $V_\alpha$:s have the Noetherian property. If Theorem 1.1   holds in this setting, strong openness can be proved in the same way as we have proved the global case. 
\section{ Example 2: The $L^2$-extension problem for general ideals.}

First we recall the classical setting of the $L^2$-extension problem for domains in $\C^n$. We let $D$ be a bounded pseudoconvex domain in $\C^n$, and $\phi$ a plurisubharmonic function in $D$. We also suppose given a linear complex subspace, $M$,  of $\C^n$ intersecting $D$. The next theorem is one version of the classical Ohsawa-Takegoshi extension theorem.
\begin{thm} There is a constant $C$, depending only on the diameter of $D$ such that for any holomorphic function $f$ on $M$ there is a holomorphic function $F$ on $D$ extending $f$ which satisfies the estimate
$$
\int_D |F|^2 e^{-\phi}\leq C \int_M |f|^2 e^{-\phi}
$$
(where on both sides we integrate with respect to Lebesgue measure). 
\end{thm}
Next we sketch a proof of this theorem along the lines of \cite{Berndtsson-Lempert}. Let $G(z)$ be a plurisubharmonic function with logarithmic singularities on $M$, such that $G\leq 0$ in $D$. If $M$ is defined by the linear equations $l_j(z)=0$ for $j= 1, 2 ,..k$ we may take $G(z)=\log \sum |l_j|^2 -C$ where $C$ is a suitable constant. We next define the subdomains of $D$, 
$$
D_t=\{z\in D; G(z)< -t\}
$$
for $t\geq 0$. Then all the domains $D_t$ are pseudoconvex and moreover the domain
$$
\D:=\{ (z, \zeta)\in D\times\H;  G(z)<-\Re\zeta\}
$$
is a pseudoconvex subdomain of $D\times\H$. Note that the domains $D_t$ are just the vertical slices of $\D$ in the sense that
$$
D_t=\{z; (z,\zeta)\in \D\}
$$
if $t=\Re \zeta$. 

It is well known that any function in $H(M)$ ( i e holomorphic on $M$) can be extended to $D$ as a holomorphic function. If we let $J(M)$ be the ideal of functions in $H(D)$ that vanish on $M$, this means that
$$
H(M) = H(D)/ J(M), 
$$
and of course also $H(M)=H(D_t)/J(M)$. This means that we get a scale of norms on $H(M)$ as the quotient norms
$$
\|f\|^2_{-t}=\min_F \int_{D_t} |F|^2 e^{-\phi},
$$
where the minimum is taken over all $F$ that extend $f$. The Ohsawa Takegoshi theorem amounts to an estimate of $\|f\|^2_0$ and the idea (first occuring in \cite{Blocki}) is to study the variation of the norms as $t$ varies. 

The proof is based on the following lemma, which is a consequence of the main result in \cite{Berndtsson}.
\begin{lma} The norms $\|f\|^2_{-\Re \zeta}$ define a vector bundle metric of non negative curvature on the vector bundle $F= H(M)\times\H$.
\end{lma}

To study the norms we look first at the dual norms on the dual bundle $E$ of $F$. By the lemma this is  a negatively curved metric and we use the following consequence of the lemma. 
\begin{cor} Let $u$ be an element of the dual space of $H(M)$ which has finite norm for $\|\cdot\|_0$.  Then  $u$ is bounded for all the norms $\|\cdot\|_t$ and if we denote by $\|u\|^2_t$ the dual norms, the function 
$$
k(t):= \log\|u\|^2_t -kt
$$
is (convex and) decreasing.
\end{cor}
\begin{proof} The convexity of $k$ follows from the discussion in the introduction: Since $\|u\|^2_{\Re\zeta} $ has negative curvature and only depends on $\Re\zeta$, $\log\|u\|^2_t$ is convex. In the situation at hand one can also verify that $k(t)$ is bounded from above as $t$ goes to infinity, and therefore the convexity implies that $k$ is decreasing.
\end{proof} 

Thus we have found that $\|u\|^2_t e^{-kt}$ is decreasing and it follows that the dual norms
$$
\|f\|^2_{-t} e^{kt}
$$
are increasing. Hence
$$
\|f\|^2_0\leq \lim_{t\to\infty}\|f\|^2_t e^{kt}.
$$
If $\phi$ is smooth and extends to a neighbourhood of $\bar D$ it is not hard to estimate the limit in the right hand side by an absolute  constant times 
$$
\int_{M\cap D} |f|^2 e^{-\phi},
$$
so we get Theorem 4.1 under these additional assumptions. But, since the constant in the estimate is absolute, the general case follows by approximation. 

This proof suggests looking at a more general situation, where  we consider more general ideals, $J$,  than $J(M)$. This problem has been studied in \cite{Popovici} and recently in \cite{Demailly}.( See also the more recent articles \cite{Hosono} and \cite{McNeal-Varolin} that appeared after this paper was submitted to the JGEA.)  The simplest such situation is when $D=\Delta$ is the unit disk in $\C$, and $J=(z^{n+1})$. Then $H(\Delta)/(z^{n+1})$ is the space of  jets of order  $n$ of holomorhic functions at the origin and we arrive at the problem to estimate the minimal weighted $L^2$-norm of all holomorphic functions $f$ in the disk with $f(0),f'(0), ...f^{(n)}(0)$ prescribed. In the more general situation, the ideal $J$ has some zero locus $M$ and one wants to extend functions on $M$ together with their jets of different orders with $L^2$-estimates.

This seems to be a rather formidable problem, but it is clear that the general lines of the argument described above still apply. We can still find a plurisubharmonic  function $G$ with logarithmic singularities on $M$ as 
$$
G(z)=\log \sum |g_j|^2 -C
$$
if we assume the ideal $J=(g_1,...g_m)$ is finitely generated. This gives again domains $D_t$ and dual vector bundles $F$ and $E$, of positive and negative curvature respectively.

The main difference, as compared to the situation in Theorem 4.1 is that there is no longer one fixed growth order. In the classical situation, we have that
$$
\|u\|^2_t \sim e^{kt}
$$
for essentially all $u$ in the dual space (at least for a dense subspace), but in the more general case, different $u$:s have different growth. In the model case of a 'fat'  point in the disk we have $N$ linearly independent vectors in the dual of $H(\Delta)/(z^{n+1})$, the sequence of derivatives of Dirac measures at the origin, $u_j=\delta_o^{(j)}$ for $0\leq j\leq n$. It is easy to verify that
$$
\|u_j\|^2_t\sim e^{jt},
$$

Thus we are precisely in the  situation in Theorem 1.1. We have a negatively curved vector bundle over $\H$, $E$, and different 'negative Lelong numbers' $\alpha$ for different vectors in the fiber of $E$. As in Theorem 1.1, let us denote by $V$ the fiber of $E$, i e the dual of $H(D)/J$. Admittedly, this is not in general of finite dimension, but if we assume that the zero locus of $J$ consist of one point (as in our model case), it has finite dimension.  We will therefore now restrict to this case, even though the general set up of the problem still applies, even for bundles of infinite rank. 

We therefore get a filtration of $V$, $(V_{\alpha_j})$ and a dual filtration $(F_{\alpha_j})$ of $V^*=H(D)/J$. 
It follows, as in the corollary, that if $u\in V_{\alpha_j}$, then 
$$
\|u\|^2_t e^{-t\alpha_j}
$$
is decreasing. In particular, if the jumping numbers are $\alpha_1< ...\alpha_n$ we always have that $\|u\|^2_te^{-t\alpha_n}$ is decreasing,  and hence $\|v\|^2_{-t}e^{t\alpha_n}$ is increasing for any $v$ in $F$. Hence we get the estimate
$$
\|v\|^2_{-0}\leq \lim_{t\to \infty}\|v\|^2_{-t}e^{t\alpha_n}.
$$
In general however, the right hand side here will  be infinite. If it is finite, it is clear that $v\in F_{\alpha_{n-1}}$ (compare the end of the proof of Theorem 1.1). In the  model case in the disk, $\alpha_j=j$ and the condition that
$v\in F_{\alpha_{n-1}}$ means that  all derivatives up to order $n-1$ of $v$ vanish at the origin. In this case, the right hand side of our estimate is indeed finite, and the estimate is sharp.

In order to proceed, we next look at  $V_{\alpha_{n-1}}$. Then
$$
\|u\|^2_t e^{-t\alpha_{n-1}}
$$
is decreasing for all $u\in V_{\alpha_{n-1}}$. The dual of this space is $F^{(1)}:=F/F_{\alpha_{n-1}}$, so
$$
\|v\|^2_{-t,1} e^{t\alpha_{n-1}}
$$
is increasing for $v\in F^{(1)}$ and
$$
\|v\|^2_{-0,1}\leq \lim_{t\to\infty}\|v\|^2_{-t,1} e^{t\alpha_{n-1}}
$$
(where $\|\cdot\|_{-t,1}$ are the quotient norms). In particular
$$
\|v\|^2_{-0}\leq \lim_{t\to\infty}\|v\|^2_{-t,1} e^{t\alpha_{n-1}}
$$
if $v$ is orthogonal to $F_{\alpha_{n-1}}$ for the norm $\|\cdot\|_{-0}$. Now, $F_{\alpha_{n-1}}$ again defines an ideal in the space of holomorphic functions in $D$; $p^{-1}(F_{\alpha_{n-1}})$, where $p:H(D)\to H(D)/J$ is the quotient map.  Hence  we have in reality just repeated the first part of the argument with $J$ replaced by $p^{-1}(F_{\alpha_{n-1}})$ , which in the model case means that we have replaced $n$ by $n-1$. Continuing in this way, we decompose $v$ in $F$,  $v=v_{n-1}+v_{n-2} +...$, where
$v_{j}\in F_{\alpha_j}$ and is orthogonal to $F_{\alpha_{j+1}}$, and can estimate $\|v\|^2_0=\|v_{n-1}\|^2_0 +...$ by the procedure above.

The main drawback with this is that the estimate we obtain depends on the orthogonal decomposition. Let us illustrate this with the model example, for $n=1$. If $f$ is holomorphic in the disk, the decomposition is $f=f_0+f_1$, where
$$
f_0=f(0)K_\phi(z,0)/K_\phi(0,0), \quad f_1=f-f(0)K_\phi(z,0)/K\phi(z,0),
$$
with $K_\phi(z,w)$ the Bergman kernel. The estimate we get is then
$$
\|f\|^2_0\leq \pi (|f_0(0)|^2 +(1/2)|f_1'(0)|^2) e^{-\phi(0)}=
$$
$$
\pi(|f(0)|^2 +(1/2)|f'(0)-f(0)(\partial\log K_\phi(z,z)/\partial z)|_{z=0}|^2)e^{-\phi(0)}.
$$

All of this can be compared with the recent work  of Demailly, \cite{Demailly}, who treats the $L^2$-extension problem for general ideals in a very general setting on compact manifolds. His work shows the existence of $L^2$-extensions in very general circumstances, but shares the feature with the discussion above that one gets an explicit estimate only for functions with the maximal vanishing order, corresponding to the space $F_{\alpha_{n-1}}$ here.

 \def\listing#1#2#3{{\sc #1}:\ {\it #2}, \ #3.}

\end{document}